\documentclass[preprint,12pt]{elsarticle}

\usepackage{amssymb}
\usepackage{amsmath}
\usepackage{amsthm}
\usepackage{amsfonts}
\usepackage{hyperref}       

\newtheorem{lem}{Lemma}[section]
\newtheorem{cor}[lem]{Corollary}
\newtheorem{thm}[lem]{Theorem}
\newtheorem{prop}[lem]{Proposition}
\newtheorem{conj}[lem]{Conjecture}
\newtheorem{rem}[lem]{Remark}
\newtheorem{example}[lem]{Example}

\newtheorem*{thm*}{Theorem}
\newtheorem*{prop*}{Proposition}
\newtheorem*{conj*}{Conjecture}
\newtheorem*{example*}{Example}
\newtheorem*{ack*}{Acknowledgments}

\newcommand{\g}{\mathfrak{g}}
\newcommand{\G}{\mathbf{G}}
\newcommand{\Complex}{\mathbb{C}}
\newcommand{\Natural}{\mathbb{N}}
\newcommand{\Mod}{\operatorname{mod}}

\journal{Journal of Algebra}

\begin{document}

\begin{frontmatter}



\title{Relative modality of elements in generalized Takiff Lie algebras} 


\author{Hugo Mathevet} 

\affiliation{organization={Laboratoire de Mathématiques de Reims (LMR) - UMR 9008},
            addressline={U.F.R. Sciences Exactes et Naturelles\\Moulin de la Housse - BP 1039}, 
            postcode={51687 REIMS cedex 2},
            country={France}}

\begin{abstract}
	 Given a natural number $m$ and a finite dimensional complex Lie algebra $\g$, the $m^{th}$ generalized Takiff Lie algebra of $\g$ is the Lie algebra $\g_m:=\g\otimes\Complex[T]/T^{m+1}$. For $n\geq m$, we define the $(m,n)$-modality of an adjoint orbit $\Omega_m$ in $\g_m$ to be the minimum codimension of an adjoint orbit in the pullback of $\Omega_m$ in $\g_n$. 
     
     In this paper, we study these invariants in generalized Takiff Lie algebras associated to a quadratic Lie algebra $\g$. We show that these invariants satisfies some concavity and hereditary properties from which we deduce that $(n-m)\chi(\g)$ is a lower bound, where $\chi(\g)$ is the index of $\g$. We prove that this lower bound is in fact an equality for a dense set of orbits, and that if $\g$ is reductive, it is always an equality when $m=0$ (and also some special orbits). We conjecture that equality holds for all $m$, $n$ when $\g$ is reductive.
\end{abstract}

\begin{keyword}
Takiff Lie algebras\sep Reductive Lie algebras\sep Adjoint action


\MSC 17B08 \sep 17B10 \sep 17B20
\end{keyword}

\end{frontmatter}


\section{Introduction}
Throughout this paper, $\g$ is a finite-dimensional Lie algebra over the complex numbers and the topology considered is the Zariski topology.

Denote by $\g_\infty$ the Lie algebra $\g\otimes\Complex[T]$ whose Lie bracket is given by

\[ [\mathbf{x},\mathbf{y}]=\sum_{i,j\in\Natural}{[x_i,y_j]\otimes T^{i+j}}\]
for all $\mathbf{x}=\sum_{n\in\Natural}{x_n\otimes T^n}$ and $\mathbf{y}=\sum_{n\in\Natural}{y_n\otimes T^n}$ in $\g_\infty$. Thus $[\mathbf{x},\mathbf{y}]=\sum_{n\in\Natural}{z_k\otimes T^k}$ where $z_k=\sum_{i=0}^{k}{[x_i,y_{k-i}]}$.

Let $m$ be a natural number and $\mathfrak{p}_m$ be the ideal $\g\otimes T^{m+1}\Complex[T]$ of $\g_\infty$. The quotient $\g_m:=\g_\infty/\mathfrak{p}_m$ is called the $m^{th}$ generalized Takiff Lie algebra of $\g$. It is a finite-dimensional complex Lie algebra of dimension $(m+1)\dim\g$. We denote $\pi_{n,m}$ the canonical projection from $\g_n$ onto $\g_m$. For $\mathbf{x}$ in $\g_\infty$, we denote by $\mathbf{x}(m)$ the projection of $\mathbf{x}$ in $\g_m$. In particular, we have $[\mathbf{x}(m),\mathbf{y}(m)]=[\mathbf{x},\mathbf{y}](m)$. Since any element of $\g_m$ is the projection of an element $\mathbf{x}=\sum_{k\in\Natural}{x_k\otimes T^k}$ of $\g_\infty$, we shall write elements of $\g_m$ as $(x_0,x_1,...,x_m)$ or $\mathbf{x}(m)$.

These Lie algebras were introduced by Takiff in \cite{Takiff1971} for $m=1$ and $\g$ semisimple giving an exemple of a non-reductive Lie algebra whose symmetric invariants form a polynomial ring. The generalization to all $m$ is due to Raïs and Tauvel in \cite{RaisTauvel1992}. Multivariable versions were studied by Macedo and Savage in \cite{MacedoSavage2019} and by  Panyushev and Yakimova in \cite{PanyushevYakimova2019}. Generalized Takiff algebras also appear in \cite{MoreauYu2016} in the study of jet schemes of the closure of nilpotent orbits of simple Lie algebras. They also arise in mathematical physics under the name of truncated current Lie algebras. In this paper we introduce and study invariants of the adjoint action relative to the choice of an element $\mathbf{x}$ in $\g_\infty$.

The modality of an algebraic group $\G$ acting on an irreducible variety $X$ is defined to be the integer $$\Mod(\G,X):=\min_{x\in X}\operatorname{codim}_X\G\cdot x,$$ where $\G\cdot x$ denotes the $\G$-orbit of $x$. For example the notion of index of $\g$, denoted $\chi(\g)$, introduced by Dixmier \cite{Dixmier1974} is equal to $\Mod(\G,\g^*)$ when $\g$ is algebraic, where $\G$ is the algebraic adjoint group of $\g$ acting on $\g^*$, the dual of $\g$, by the coadjoint action. In the semisimple case it is equal to the rank of the Lie algebra.

Assume that $\g$ is algebraic and let $\G_k$ be the algebraic adjoint group of $\g_k$ for all positive integer $k$. For $\mathbf{x}\in\g_\infty$ and $0\leq m\leq n$, $\Omega_{m,n}(\mathbf{x}):=\pi_{n,m}^{-1}(\G_m\cdot\mathbf{x}(m))$ is an irreducible $\G_n$-stable subvariety of $\g_n$. We are interested in the following integer:

\[\Mod_{m,n}(\mathbf{x}):=\Mod(\G_n,\Omega_{m,n}(\mathbf{x})).\]

 We call this integer the $(m,n)$-modality of the element $\mathbf{x}$, note that it only depends on the $\G_m$-orbit of $\mathbf{x}(m)$. This quantity can be expressed in terms of dimension of centralizers as follows: 

\[\Mod_{m,n}(\mathbf{x})=\min_{\substack{\mathbf{y}\in\g_\infty\\\mathbf{y}(
    m)=\mathbf{x}(m)}}{\dim\g_n^{\mathbf{y}(n)}-\dim\g_m^{\mathbf{x}(m)}},\]
where for an element $z$ in a Lie algebra $\mathfrak{h}$, $\mathfrak{h}^{z}$ denotes the centralizer of $z$ in $\mathfrak{h}$. This second definition does not require $\g$ to be algebraic and will be the one used in practice.

The first explicit examples are for all integers $n$ and all elements $\mathbf{x}$ of $\g_\infty$, $\Mod_{n,n}(\mathbf{x})=0$ and if the Lie algebra $\g$ is abelian then, for all $m$ and $n$ natural numbers with $m\leq n$ and $\mathbf{x}$ elements of $\g_\infty$, $\Mod_{m,n}(\mathbf{x})=(n-m)\dim\g$.   

In this paper, the Lie algebra $\g$ is assumed to be quadratic, meaning that its adjoint and coadjoint representation are isomorphic. The notion of quadratic Lie algebras encompasses many Lie algebras but most noticeably reductive and semisimple Lie algebras, we refer the reader to Section \ref{sec:quad} for more examples. Recall that an element $x$ of $\g$ is called regular if its centralizer is of minimal dimension, here we have $\dim\g^x=\chi(\g)$. The dimensions of adjoint centralizers have same parity (see \cite{Dixmier1974} 1.11.2), thus we shall call an element $x$ in $\g$ subregular if $\dim\g^x=\chi(\g)+2$.

Our motivation to study $(m,n)$-modality is to obtain a better understanding of adjoint orbits in generalized Takiff Lie algebras. This work is very much inspired by and relies on a theorem of Raïs and Tauvel and a conjecture of Elashvili, proved by Yakimova \cite{Yakimova2006} for classical reductive Lie algebras, and de Graaf \cite{deGraaf2007} for exceptional ones with an almost case free proof by Charbonnel and Moreau \cite{CharbonnelMoreau2010}.
\begin{thm}[{\cite[Théorème 2.8]{RaisTauvel1992}}]
\label{thm:RT}
    Let $m$ be a natural number, we have $\chi(\g_m)=(m+1)\chi(\g)$, and $(x_0,x_1,...,x_m)$ is a regular element of $\g_m$ if and only if $x_0$ is regular in $\g$.
\end{thm}
\begin{thm}[Elashvili's conjecture]
\label{conj:Elashvili}
    Let $\g$ be a reductive Lie algebra and $x$ an element of $\g$, then:
    $$\chi(\g^x)=\chi(\g).$$
\end{thm}
The notion of $(m,n)$-modality will provide a common framework for these results. For example, by proving that $\Mod_{0,1}(\mathbf{x})=\chi(\g^{x_0})$ for all $\mathbf{x}$ in $\g_\infty$, we show that Theorem \ref{conj:Elashvili} is equivalent to 
\begin{thm}
\label{thm: 01}
    Let $\g$ be a reductive Lie algebra and $\mathbf{x}$ an element of $\g_\infty$, then:
    $$\Mod_{0,1}(\mathbf{x})=\chi(\g).$$
\end{thm}
One can then combine such results together with reformulations of structural theorems on generalized Takiff Lie algebras like Theorem \ref{thm:RT} to prove various consequences and generalizations.
\paragraph{Main results}
The main tool we used to study properties of $(m,n)$-modality is the concavity of the sequences $(\dim\g_m^{\mathbf{x}(m)})_{m\in\Natural}$ for all elements $\mathbf{x}$ of $\g_\infty$ proven in Proposition \ref{prop:adjoint concavity}. A first consequence is a lower bound for the $(m,n)$-modality.

\begin{thm}
    \label{thm:lower bound}
    Let $m$ and $n$ be two natural numbers with $m\leq n$ and $\mathbf{x}$ an element of $\g_\infty$, then
    $$\Mod_{m,n}(\mathbf{x})\geq (n-m)\chi(\g).$$
\end{thm}

We then prove that the condition of attaining this lower bound enjoys some nice hereditary properties.

\begin{prop}
\label{prop:extension}
     Let $m$ and $n$ be two natural numbers with $m<n$ and $\mathbf{x}$ an element of $\g_\infty$ such that $\Mod_{m,n}(\mathbf{x})=(n-m)\chi(\g)$, then for all natural numbers $n'\geq m$
     \[\Mod_{m,n'}(\mathbf{x})=(n'-m)\chi(\g).\]
\end{prop}

\begin{prop}
\label{prop:full ext}
Let $\mathbf{x}$ be an element of $\g_\infty$ and $m$, $n$ two natural numbers with $m<n$. Then, for all elements $\mathbf{y}$ of $\g_\infty$ such that $\mathbf{x}(m)=\mathbf{y}(m)$ and $\dim\g_n^{\mathbf{y}(n)}=(n-m)\chi(\g)+\dim\g_m^{\mathbf{x}(m)}$, we have that for all pairs of integers $(m',n')$ with $m\leq m'\leq n'$
   $$\Mod_{m',n'}(\mathbf{y})=(n'-m')\chi(\g).$$
\end{prop}

Using these results, we proceed to show in Section \ref{sec:sharp} that the above lower bound is exact in a lot of cases. In particular, that Theorem \ref{thm:RT} implies
\begin{thm}
    \label{thm:reached reg}
    Let $\mathbf{x}$ be an element of $\g_\infty$ such that $x_0$ is regular, then for all pairs of natural numbers $(m,n)$ with $m\leq n$,
    \[\Mod_{m,n}(\mathbf{x})=(n-m)\chi(\g).\]
\end{thm}
In addition, we prove
\begin{prop}
\label{prop: reached subreg}
    Let $\mathbf{x}$ be an element of $\g_\infty$ such that $x_0$ is subregular, then the following are equivalent:
    \begin{enumerate}
        \item for all pairs of natural numbers $(m,n)$ with $m\leq n$ we have that $\Mod_{m,n}(\mathbf{x})=(n-m)\chi(\g)$,
        \item there exists a pair of natural numbers $(m,n)$ with $m<n$ such that $\Mod_{m,n}(\mathbf{x})=(n-m)\chi(\g)$,
        \item the centralizer $\g^{x_0}$ is not an abelian Lie algebra.
    \end{enumerate}
\end{prop}

We then restrict ourselves to the case where $\g$ is reductive for which our results yields
\begin{thm}
\label{thm:red reg/sub}
    Let $\g$ be a reductive Lie algebra and $\mathbf{x}$ be an element of $\g_\infty$. If $x_0$ is regular or subregular, then for any pair of natural numbers $(m,n)$ with $m\leq n$, we have
        \[\Mod_{m,n}(\mathbf{x})=(n-m)\chi(\g).\]
\end{thm}

Theorem \ref{thm: 01} together with the results on heredity, allows to prove the following generalization
\begin{thm}
    \label{thm:main}
    Let $\g$ be a reductive Lie algebra and $\mathbf{x}$ an element of $\g_\infty$. For all natural numbers $n$, we have
    $$\Mod_{0,n}(\mathbf{x})=n\chi(\g).$$
\end{thm}

Observe that the right-hand side of the equality is independent of $\mathbf{x}$, this mirrors Elashvili's conjecture and can be seen as a slight extension of it.

Finally, we focus on the case where the Lie algebra is semisimple to further our results and give a reduction result from the Jordan decomposition of an element.

Supported by our results, we formulate the following generalization of Elashvili's conjecture

\begin{conj}
\label{conj: main}
    Let $\g$ be a reductive Lie algebra and $\mathbf{x}$ an element of $\g_\infty$. For all natural numbers $m$ and $n$ with $m\leq n$, we have
    $$\Mod_{m,n}(\mathbf{x})=(n-m)\chi(\g).$$
\end{conj}

\begin{rem}
 \label{count: quadratic}
    Conjecture \ref{conj: main} and in particular Theorem \ref{conj:Elashvili} are false if reductive is replaced by quadratic. The following construction, given in \cite{RaisTauvel1992} (4.6 Remarques.c.ii) for another purpose, yields a counter-example.
    
    Consider the exceptional simple Lie algebra $\mathfrak{s}$ of type $G_2$ together with its Chevalley basis
    $$\{h_\alpha,h_\beta,e_{\pm\alpha},e_{\pm\beta},e_{\pm(\beta+\alpha)},e_{\pm(\beta+2\alpha)},e_{\pm(\beta+3\alpha)},e_{\pm(2\beta+3\alpha)}\}.$$
    Let $\g$ be the Takiff algebra of $\mathfrak{s}$, it is quadratic as $\mathfrak{s}$ is (see Section \ref{sec:quad}). Let $x=(e_\beta+e_{\beta+3\alpha},e_{-(2\beta+3\alpha)})$, it follows from a direct computation that $\g^{x}$ is abelian of dimension $6$. Yet, we know by Theorem \ref{thm:RT} that $\chi(\g)=2\chi(\mathfrak{s})=4$ so $x$ is subregular, hence by Proposition \ref{prop: reached subreg},  $\Mod_{m,n}(\mathbf{x})=6(n-m)>\chi(\g)(n-m)$ for all $m\leq n$ and elements $\mathbf{x}$ of $\g_\infty$.
\end{rem}

Note that a lot of the results in this paper admits a coadjoint reformulation for which the quadraticity assumption can be dropped as for example Theorem \ref{thm:RT} (See Lemma \ref{lem: inter K duality}).
\hspace{0px}

We shall conserve the notations introduced above in the rest of the paper. 

\section{Concavity in a generalized Takiff algebra}
\label{sec:quad}
The goal of this section is to prove that for any element $\mathbf{x}$ of $\g_\infty$, the sequence of integers $(\dim\g_m^{\mathbf{x}(m)})_{m\in\Natural}$ is concave. Let us recall some basic properties of quadratic Lie algebras

A bilinear form $K$ on a Lie algebra $\g$ is said to be invariant if it satisfies:
\[\text{for all } x,y,z\text{ in }\g:~K([x,y],z)=K(x,[y,z]).\]

A Lie algebra is quadratic if and only if it admits a non-degenerate, symmetric, invariant bilinear form. Throughout this paper $K$ will denote a non-degenerate, symmetric, invariant bilinear form on $\g$. Let us give a few examples that we will use later on.

\begin{example}
\label{example: list}
    \begin{enumerate}
        \item A semisimple Lie algebra is quadratic because its Killing form is non-degenerate. 
        \item An abelian Lie algebra is quadratic as any non-degenerate symmetric bilinear form is invariant.
        \item The product of two quadratic Lie algebras is quadratic. In particular, a reductive Lie algebra is quadratic.
        \item A nilpotent (non-abelian) Lie algebra can be quadratic. Take $\mathfrak{n}$ to be the Lie algebra with basis $(a_1,a_2,a_3,b_1,b_2,b_3)$ and non trivial brackets $[a_i,a_j]=b_k$ for $\{i,j,k\}=\{1,2,3\}$ and $i<j$. The symmetric bilinear form $K$ on $\mathfrak{n}$ given on the basis by the non-trivial couplings $K(a_i,b_i)=1$ for $i$ in $\{1,2,3\}$ is invariant and non-degenerate.
        \item A solvable (non-nilpotent) Lie algebra can be quadratic. Take $\mathfrak{s}$ to be the Lie algebra with basis $(a_1,a_2,a_3,b)$ and non-trivial brackets $[a_1,a_2]=a_3,~[a_1,a_3]=-a_2,~[a_2,a_3]=b$. The symmetric bilinear form $K$ on $\mathfrak{s}$ given on the basis by the non-trivial couplings $K(a_1,b)=K(a_2,a_2)=K(a_3,a_3)=1$ is invariant and non-degenerate.
        \item If $\g$ is any Lie algebra (not necessarly quadratic), the semi-direct product of $\g$ by $\g^*$ with respect to the coadjoint action is quadratic as the bilinear form $K$ on $\g\ltimes\g^*$ defined by $K(x+\alpha,y+\beta)=\alpha(y)+\beta(x)$, where $a,b\in\g$ and $\alpha,\beta\in\g^*$, is invariant and non-degenerate.
    \end{enumerate}
\end{example}

The bilinear form $K$ can be extended to a symmetric non-degenerate invariant bilinear form on $\g_m$, still denoted $K$, defined as follows:
\begin{center}
$K(\mathbf{x}(m),\mathbf{y}(m)):=\sum_{k=0}^{m}{K(x_k,y_{m-k})}$ for all $\mathbf{x}(m),\mathbf{y}(m)$ in $\g_m$.
\end{center}
In particular, $\g_m$ is also quadratic.

We identify the dual $\g_m^*$ of $\g$, as a vector space with the product of $m+1$ copies of $\g^*$ as follows, for an element $\mathbf{f}(m):=(f_0,...,f_m)$ with $f_0,...,f_m\in\g^*$, we set:

\begin{center}
    $\mathbf{f}(m)(\mathbf{x}(m)):=\sum_{k=0}^{m}{f_k(x_k)}$ for all $\mathbf{x}(m)$ in $\g_m$.
\end{center}

For a Lie algebra $\mathfrak{h}$ and a linear form $\phi$ on $\mathfrak{h}$, we denote $\mathfrak{h}^{\phi}=\{x\in\mathfrak{h}~|~\phi([x,\mathfrak{h}])=0\}$ the centralizer of $\phi$ under the coadjoint action of $\mathfrak{h}$, it is a Lie subalgebra of $\mathfrak{h}$. The index of $\mathfrak{h}$ is then defined to be $\chi(\mathfrak{h}):=\min_{f\in\mathfrak{h}^*}\dim\mathfrak{h}^{f}$. When $\mathfrak{h}$ is algebraic this definition is compatible with the one given in the introduction in term of modality. If $\mathfrak{h}$ is quadratic with symmetric non-degenerate invariant bilinear form $K$, we have $\mathfrak{h}^x=\mathfrak{h}^{K(x,\cdot)}$ for all $x$ in $\mathfrak{h}$. The following lemma describes how the symmetric non-degenerate invariant bilinear form on $\g$ and $\g_m$ interact with this duality.

\begin{lem}
\label{lem: inter K duality}
    Let $\mathbf{x}$ be an element of $\g_\infty$ and $m$ a natural number. For all integers $i$ with $0\leq i\leq m$, set $f_i:=K(x_{m-i},\cdot)\in\g^*$ and $\mathbf{f}(m):=(f_0,f_1...,f_m)\in\g_m^*$, then
    \[\mathbf{f}(m)=K(\mathbf{x}(m),\cdot).\]
\end{lem}
\begin{proof}
    Let $\mathbf{y}(m)$ be an element of $\g_m$,
    \begin{align*}
        \mathbf{f}(m)(\mathbf{y}(m))&=\sum_{k=0}^{m}{f_k(y_{k})}\\
        &=\sum_{k=0}^{m}{K(x_{m-k},y_{k})}\\
        &=K((x_0,...,x_m),(y_0,...,y_m))\\
        &=K(\mathbf{x}(m),\mathbf{y}(m)).
    \end{align*}
    The result follows as this equality holds for an arbitrary $\mathbf{y}(m)\in\g_m$.
\end{proof}

Lemma \ref{lem: inter K duality} together with Lemme 2.7 \cite{RaisTauvel1992} yields

\begin{lem}
\label{lem: adjoint RT tech}
    Let $\mathbf{x}$ be an element of $\g_\infty$ and $m$ be an integer with $m\geq2$. Let $\mathbf{x}'$ be the element of $\g_\infty$ satisfying that for all natural numbers $i$, $x'_i=(i+1)x_{i+1}$, then
    \[\dim\g_m^{\mathbf{x}(m)}=\dim\g_{m-1}^{\mathbf{x}(m-1)}+\dim((\g_{m-1}^{\mathbf{x}(m-1)})^{\mathbf{h}(m-1)})-\dim\g_{m-2}^{\mathbf{x}(m-2)},\]
    where $\mathbf{h}(m-1)$ is the restriction to ${\g_{m-1}^{\mathbf{x}(m-1)}}$ of $K(\mathbf{x'}(m-1),\cdot)$.
\end{lem}

\subsection{Concavity of the dimension of centralizers in a generalized Takiff algebra}

Recall that a sequence of real numbers $(u_m)_{m\in\Natural}$ is said to be concave if for all positive integer $m$, $u_{m+2}+u_{m}\leq2u_{m+1}$.

We shall make extensive use of the following discrete version of the ``Slopes lemma'' for concave sequences.
\begin{lem}[Slopes lemma]
    Let $(u_n)_{n\in\Natural}$ be a concave sequence of real numbers and $l,m,n$ natural numbers with $l< m< n$, then
    $$\frac{u_n-u_m}{n-m}\leq\frac{u_n-u_l}{n-l}\leq\frac{u_m-u_l}{m-l}.$$
\end{lem}

The main result of the section then follows almost immediately from Lemma \ref{lem: adjoint RT tech}.

\begin{prop}
\label{prop:adjoint concavity}
    For all elements $\mathbf{x}$ of $\g_\infty$, the sequence $(\dim\g_m^{\mathbf{x}(m)})_{m\in\Natural}$ is concave.
\end{prop}

\begin{proof}
    Let $\mathbf{x}$ be an element of $\g_\infty$ and $m$ an integer with $m\geq 2$, Lemma \ref{lem: adjoint RT tech} states that for a certain linear form $\eta$ on $\g_{m-1}^{\mathbf{x}(m-1)}$, we have
     \[\dim\g_m^{\mathbf{x}(m)}=\dim\g_{m-1}^{\mathbf{x}(m-1)}+\dim((\g_{m-1}^{\mathbf{x}(m-1)})^{\eta})-\dim\g_{m-2}^{\mathbf{x}(m-2)}.\]
    It follows from the inclusion $(\g_{m-1}^{\mathbf{x}(m-1)})^{\eta}\subset\g_{m-1}^{\mathbf{x}(m-1)}$ that
    \[\dim\g_m^{\mathbf{x}(m)}\leq2\dim\g_{m-1}^{\mathbf{x}(m-1)}-\dim\g_{m-2}^{\mathbf{x}(m-2)}.\]
    As the results holds for any $m\geq2$, concavity of the sequence $(\dim\g_m^{\mathbf{x}(m)})_{m\in\Natural}$ ensues.
\end{proof}
In addition, let us prove an ``$m=1$'' version of Lemma \ref{lem: adjoint RT tech} as it will be of interest in later sections.
\begin{lem}
\label{lem:m=1}
    Let $(x_0,x_1)$ be an element of $\g_1$, then
    \[\dim\g_1^{(x_0,x_1)}=\dim\g^{x_0}+\dim(\g^{x_0})^{K(x_1,\cdot)|_{{\g^{x_0}}}}.\]
    In particular: $\dim\g_1^{(x_0,x_1)}\leq2\dim\g^{x_0}$.
\end{lem}
\begin{proof}
    First, note that $(y_0,y_1)$ is in $\g_1^{(x_0,x_1)}$ if and only if $[x_0,y_0]=0$ and $[x_0,y_1]+[x_1,y_0]=0$. Fixing $y_0$ an element of $\g$, we deduce that
    there exists $y_1$ in $\g$ such that $(y_0,y_1)$ is in $\g_1^{(x_0,x_1)}$ if and only if $y_0$ is in $\g^{x_0}$ and $[x_1,y_0]$ is in $[x_0,\g]$.

    By invariance of $K$, we have that $[x_0,\g]$ is a subspace of the orthogonal $(\g^{x_0})^\perp$ of $\g^{x_0}$ with respect to $K$. By  non-degeneracy of $K$, we have that $\dim(\g^{x_0})^\perp=\dim\g-\dim\g^{x_0}=\dim[x_0,\g]$, thus the inclusion is an equality. Hence there exists $y_1$ in $\g$ such that $(y_0,y_1)$ is in $\g_1^{(x_0,x_1)}$ if and only if $ y_0$ is in $\g^{x_0}$ and $K([x_1,y_0],\g^{x_0})=0$. By invariance of $K$, the last condition is equivalent to $K(x_1,[y_0,\g^{x_0}])=0$. Hence there exists $y_1$ in $\g$ such that $(y_0,y_1)$ is in $\g_1^{(x_0,x_1)}$ if and only if $y_0$ is in $(\g^{x_0})^{K(x_1,\cdot)|_{{\g^{x_0}}}}$.

    We conclude by observing that if the space $\{y_1\in\g~|~\ (y_0,y_1)\in\g_1^{(x_0,x_1)}\}$ is non-empty, it is of dimension $\dim\g^{x_0}$.
\end{proof}

\section{Lower bound and heredity results}

\subsection{Lower bound for $(m,n)$-modality}
In this subsection, we shall prove Theorem \ref{thm:lower bound} announced in the introduction using Theorem \ref{thm:RT} together with the concavity proven in Proposition \ref{prop:adjoint concavity}.
\begin{proof}[Proof of Theorem \ref{thm:lower bound}]
    Assume that $\Mod_{m,n}(\mathbf{x})<(n-m)\chi(\g)$, necessarily $m<n$, and let $\mathbf{y}$ be an element of $\g_\infty$ such that $\mathbf{y}(m)=\mathbf{x}(m)$ and $\dim\g_n^{\mathbf{y}(n)}=\dim\g_m^{\mathbf{x}(m)}+\Mod_{m,n}(\mathbf{x})$. 

    By Proposition \ref{prop:adjoint concavity} the sequence $(\dim\g_n^{\mathbf{y}(n)})_{n\in\Natural}$ is concave, as such we can apply the ``Slopes Lemma'' to obtain that for all $n'>n,$
    $$\dim\g_{n'}^{\mathbf{y}(n')}\leq\dim\g_n^{\mathbf{y}(n)}+(n'-n)\frac{\dim\g_n^{\mathbf{y}(n)}-\dim\g_m^{\mathbf{y}(m)}}{n-m}.$$
    This, by our choice of $\mathbf{y}$, implies that
    $$\dim\g_{n'}^{\mathbf{y}(n')}\leq\dim\g_n^{\mathbf{y}(n)}+(n'-n)\frac{\Mod_{m,n}(\mathbf{x})}{n-m}.$$
    As a function of $n'$, the right-hand side is affine with slope $\frac{\Mod_{m,n}(\mathbf{x})}{n-m}<\chi(\g)$ by our assumption. Thus, for $n'$ big enough, we would eventually have $\dim\g_{n'}^{\mathbf{y}(n')}<(n'+1)\chi(\g)$, which is contradictory to Theorem \ref{thm:RT}. Hence, the inequality follows.
\end{proof}

\subsection{Heredity properties}
We shall now prove Propositions \ref{prop:extension} and \ref{prop:full ext} using this lower bound.

\begin{proof}[Proof of Proposition \ref{prop:extension}]

    Let $n'$ be a natural number greater or equal to $m$ and $\mathbf{y}$ be an element of $\g_\infty$ such that $\mathbf{y}(m)=\mathbf{x}(m)$ and $\dim\g_n^{\mathbf{y}(n)}=\Mod_{m,n}(\mathbf{x})+\dim\g_m^{\mathbf{x}(m)}$.
    As the case where $n=n'$ is trivial, the proof splits into two cases $n'<n$ and $n<n'$.
    \paragraph{Case $n'<n$} 
    By Proposition \ref{prop:adjoint concavity}, we have that $(\dim\g_k^{\mathbf{y}(k)})_{k\in\Natural}$ is concave. The ``Slopes lemma'' then implies that
    $$\frac{\dim\g_{n}^{\mathbf{y}(n)}-\dim\g_{n'}^{\mathbf{y}(n')}}{n-n'}\leq\frac{\dim\g_{n}^{\mathbf{y}(n)}-\dim\g_{m}^{\mathbf{y}(m)}}{n-m}.$$
    Therefore we have $\Mod_{n',n}(\mathbf{y})\leq\dim\g_{n}^{\mathbf{y}(n)}-\dim\g_{n'}^{\mathbf{y}(n')}\leq(n-n')\chi(\g)$.
    Together with Theorem \ref{thm:lower bound}, we have the equality $\dim\g_{n}^{\mathbf{y}(n)}-\dim\g_{n'}^{\mathbf{y}(n')}=(n-n')\chi(\g)$. Since $\mathbf{x}(m)=\mathbf{y}(m)$, we have
    \begin{align*}
        \Mod_{m,n'}(\mathbf{x})&\leq\dim\g_{n'}^{\mathbf{y}(n')}-\dim\g_{m}^{\mathbf{y}(m)}\\
        &=(\dim\g_{n}^{\mathbf{y}(n)}-\dim\g_{m}^{\mathbf{y}(m)})-(\dim\g_{n}^{\mathbf{y}(n)}-\dim\g_{n'}^{\mathbf{y}(n')})\\
        &=(n'-m)\chi(\g).
    \end{align*}
     Again, using Theorem \ref{thm:lower bound}, we obtain $\Mod_{m,n'}(\mathbf{x})=(n'-m)\chi(\g).$
    \paragraph{Case $n<n'$} By Proposition \ref{prop:adjoint concavity}, we have that $(\dim\g_k^{\mathbf{y}(k)})_{k\in\Natural}$ is concave. By the ``Slopes lemma'' together with the fact that $\mathbf{x}(m)=\mathbf{y}(m)$, this implies that
    $$\frac{\dim\g_{n'}^{\mathbf{y}(n')}-\dim\g_{m}^{\mathbf{y}(m)}}{n'-m}\leq\frac{\dim\g_{n}^{\mathbf{y}(n)}-\dim\g_{m}^{\mathbf{y}(m)}}{n-m}=\chi(\g)$$
    Hence,
    $$\Mod_{m,n'}(\mathbf{x})\leq\dim\g_{n'}^{\mathbf{y}(n')}-\dim\g_{m}^{\mathbf{y}(m)}\leq(n'-m)\chi(\g).$$
    Together with Theorem \ref{thm:lower bound} this yields $\Mod_{m,n'}(\mathbf{x})=(n'-m)\chi(\g).$   
\end{proof}
\begin{rem}
    It follows from Proposition 1.8 that if we fix a positive integer $m$ and an element $\mathbf{x}$ in $\mathfrak{g}_{\infty}$, then $\Mod_{m,m+1}(\mathbf{x})=\chi(\mathfrak{g})$ if and only if $\Mod_{m,n}(\mathbf{x})=(m-n)\chi(\mathfrak{g})$ for all $n > m$. Thus to prove Conjecture 1.5, it suffices to prove $\Mod_{m,m+1} (\mathbf{x}) = \chi (\mathfrak{g})$ for all $m$.
\end{rem}

In addition to Proposition \ref{prop:extension} proving that, given $\mathbf{x}$ in $\g_\infty$, $\Mod_{m,n}(\mathbf{x})=(n-m)\chi(\g)$ is an hereditary property on $n$, we prove Proposition \ref{prop:full ext} that shows that if $\Mod_{m,n}(\mathbf{x})=(n-m)\chi(\g)$ then, $\Mod_{m,n}(\mathbf{y})$ also attains the lower bound for a larger class of elements $\mathbf{y}$ in $\g_\infty$.

\begin{proof}[Proof of Proposition \ref{prop:full ext}]
    Fix $\mathbf{y}$ in $\g_\infty$ verifying the hypotheses. The core of the proof is to show that $\dim\g_k^{\mathbf{y}(k)}=\dim\g_m^{\mathbf{y}(m)}+(k-m)\chi(\g)$ for all $k\geq m$.
    
    First, suppose that for $k=m+1$. By Proposition \ref{prop:adjoint concavity}, the sequence $(\dim\g_l^{\mathbf{y}(l)})_{l\in\Natural}$ is concave. Hence by the ``Slopes lemma''
    \[\frac{\dim\g_n^{\mathbf{y}(n)}-\dim\g_{m+1}^{\mathbf{y}(m+1)}}{n-(m+1)}\leq\frac{\dim\g_n^{\mathbf{y}(n)}-\dim\g_{m}^{\mathbf{y}(m)}}{n-m}=\chi(\g).\]
    Now as $\mathbf{y}(n)$ is in $\{\mathbf{z}(n)\in\g_n~|~\mathbf{z}(m+1)=\mathbf{y}(m+1)\}$, we have
    \[\Mod_{m+1,n}(\mathbf{y})\leq\dim\g_n^{\mathbf{y}(n)}-\dim\g_{m+1}^{\mathbf{y}(m+1)}\leq(n-(m+1))\chi(\g).\]
    Together with Theorem \ref{thm:lower bound}, this implies the equality 
    \[\dim\g_n^{\mathbf{y}(n)}-\dim\g_{m+1}^{\mathbf{y}(m+1)}=(n-(m+1))\chi(\g).\]
    Now, by hypothesis on the dimension of $\g_n^{\mathbf{y}(n)}$, we have
    \[\dim\g_{m+1}^{\mathbf{y}(m+1)}=\dim\g_m^{\mathbf{y}(m)}+\chi(\g).\]
    We shall now extend the equality to all $k\geq m$. By the concavity of the sequence $(\dim\g_k^{\mathbf{y}(k)})_{k\in\Natural}$, we know that for all natural numbers $k$
    \[\dim\g_{k+2}^{\mathbf{y}(k+2)}-\dim\g_{k+1}^{\mathbf{y}(k+1)}\leq\dim\g_{k+1}^{\mathbf{y}(k+1)}-\dim\g_{k}^{\mathbf{y}(k)}.\]
    It follows by a straightforward induction that for all $k\geq m$
    \[\dim\g_{k+1}^{\mathbf{y}(k+1)}-\dim\g_{k}^{\mathbf{y}(k)}\leq\dim\g_{m+1}^{\mathbf{y}(m+1)}-\dim\g_{m}^{\mathbf{y}(m)}=\chi(\g).\]
    As $\dim\g_{k+1}^{\mathbf{y}(k+1)}-\dim\g_{k}^{\mathbf{y}(k)}\geq\Mod_{k,k+1}(\mathbf{y})$, Theorem \ref{thm:lower bound} implies that for all $k\geq m$,
    \[\dim\g_{k+1}^{\mathbf{y}(k+1)}-\dim\g_{k}^{\mathbf{y}(k)}=\chi(\g).\]
    Again, a quick induction gives that for all $k\geq m$,
    \[\dim\g_{k}^{\mathbf{y}(k)}=(k-m)\chi(\g)+\dim\g_{m}^{\mathbf{y}(m)}.\]
    From this we deduce that for all pairs of integers $(m',n')$ with $m\leq m'\leq n'$, we have
    \[\dim\g_{n'}^{\mathbf{y}(n')}-\dim\g_{m'}^{\mathbf{y}(m')}=(n'-m')\chi(\g).\]
    As the left-hand side is larger than or equal to $\Mod_{m',n'}(\mathbf{y})$, Theorem \ref{thm:lower bound} implies the desired equality.
\end{proof}

\subsection{Shift property}
We end this section with a result on a ``shifting'' property of $(m,n)$-modality. We start with a technical lemma on centralizers in generalized Takiff Lie algebras whose proof is a straightforward verification.
\begin{lem}
    Let $m$ and $k$ be two natural numbers with $k\geq1$ and $\mathbf{z}$ an element of $\g_\infty$. Let $\mathbf{z}'(m+k)=(0,...,0,z_0,...,z_m)$ in $\g_{m+k}$, then
    \[\g_{m+k}^{\mathbf{z}'(m+k)}=\{(y_0,y_1,...,y_{m+k})\in\g_{m+k}~|~(y_0,...,y_m)\in\g_m^{\mathbf{z}(m)}\}.\]
    In particular, we have equality $\dim\g_{m+k}^{\mathbf{z}'(m+k)}=\dim\g_m^{\mathbf{z}(m)}+k\dim\g$.
\end{lem}

Using this lemma, we show the ``shifting'' property of the $(m,n)$-modality.

\begin{prop}
\label{prop:shift}
    Let $\mathbf{x}$ be an element of $\g_\infty$, $k$ a strictly positive integer and $\mathbf{x'}$ the element of $\g_\infty$ such that  $\mathbf{x'}(k-1)=(0,...,0)$ and $x'_{k+i}=x_i$ for all natural numbers $i$. Then, for all natural numbers $m$ and $n$ such that $m\leq n$, we have 
    \[\Mod_{m,n}(\mathbf{x})=\Mod_{m+k,n+k}(\mathbf{x'}).\]
\end{prop}
\begin{proof}
    Consider an element $\mathbf{z}$ of $\g_\infty$ such that $\mathbf{z}(m)=\mathbf{x}(m)$ and $\dim\g_n^{\mathbf{z}(n)}-\dim\g_m^{\mathbf{x}(m)}=\Mod_{m,n}(\mathbf{x})$ and let $\mathbf{z}'(n+k)=(0,...,0,z_0,...,z_n)$ in $\g_{n+k}$. By the previous lemma, we have that
    \begin{align*}
        \dim\g_{n+k}^{\mathbf{z}'(n+k)}&=\dim\g_n^{\mathbf{z}(n)}+k\dim\g\\
        &=\Mod_{m,n}(\mathbf{x})+\dim\g_m^{\mathbf{x}(m)}+k\dim\g\\
        &=\Mod_{m,n}(\mathbf{x})+\dim\g_{n+k}^{\mathbf{x}'(m+k)}.
    \end{align*}
    As $\mathbf{z}'(m+k)=\mathbf{x}'(m+k)$, we have $\Mod_{m,n}(\mathbf{x})\geq\Mod_{m+k,n+k}(\mathbf{x}')$.
    
    Conversely, let $\mathbf{z}'\in\g_\infty$ be such that $\mathbf{z}'(m+k)=\mathbf{x}'(m+k)$ and $\dim\g_{n+k}^{\mathbf{z}'(n+k)}=\dim\g_{m+k}^{\mathbf{x}'(m+k)}+\Mod_{m+k,n+k}(\mathbf{x}')$. Let $\mathbf{z}\in\g_\infty$ be such that $z_i=z'_{i+k}$ for all natural numbers $i$. We have that $\mathbf{z}(m)=\mathbf{x}(m)$ and therefore, by the previous lemma,
    \begin{align*}
        \Mod_{m,n}(\mathbf{x})&\leq\dim\g_n^{\mathbf{z}(n)}-\dim\g_m^{\mathbf{x}(m)}\\
        &=(\dim\g_n^{\mathbf{z}(n)}+k)-(\dim\g_m^{\mathbf{x}(m)}+k)\\
        &=\dim\g_{n+k}^{\mathbf{z}'(m+k)}-\dim\g_{n+k}^{\mathbf{x}'(m+k)}\\
        &=\Mod_{m+k,n+k}(\mathbf{x}').
    \end{align*}
    This concludes the proof. 
\end{proof}

\section{Cases for which the lower bound is sharp}
\label{sec:sharp}
In this section, we give cases for which the bound given in Theorem \ref{thm:lower bound} is sharp. We have seen in the introductory example that it was the case for abelian Lie algebras as $\dim\g=\chi(\g)$.
\subsection{Regular and subregular cases in quadratic Lie algebras}
\label{sub:reg}
We begin by proving our claim that Theorem \ref{thm:RT} implies Theorem \ref{thm:reached reg}.
\begin{proof}[Proof of Theorem \ref{thm:reached reg}]
    Let $\mathbf{x}$ be an element of $\g_\infty$ such that $x_0$ is a regular element of $\g$. From Theorem \ref{thm:RT}, it follows that $\dim\g_1^{\mathbf{x}(1)}=2\chi(\g)=\chi(\g)+\dim\g^{x_0}$. Therefore, we have $\Mod_{0,1}(\mathbf{x})=\chi(\g)$ and we conclude by applying Proposition \ref{prop:full ext} to $\mathbf{x}$ with $\mathbf{y}=\mathbf{x}$ and $(m,n)=(0,1)$.
\end{proof}

\begin{rem}
   From this Theorem, it follows that given an integer $k$ the projection of $\{\mathbf{x}\in\g_\infty~|~\Mod_{m,n}(\mathbf{x})=(n-m)\chi(\g)\text{ for all } n\geq m\}$ on $\g_k$ is dense as it contains $\{\mathbf{x}(k)\in\g_k~|~x_0 \text{ is regular in }\g\}$ which is dense being open and non-empty.
\end{rem}

\begin{example}
\label{example: s}
    Let $\mathfrak{s}$ be the Lie algebra as in Example \ref{example: list}. We easily check that every non-trivial element of $\mathfrak{s}$ is regular with centralizer of dimension $2$.
    
    We conclude by Theorem \ref{thm:reached reg} that for all elements $\mathbf{x}$ of $\mathfrak{s}_\infty$ such that $x_0\neq0$ and all pairs of natural numbers $(m,n)$ with $m\leq n$, $\Mod_{m,n}(\mathbf{x})=(n-m)\chi(\mathfrak{s})=2(n-m)$. It follows easily from Proposition \ref{prop:shift} that the equality is still true even if $x_0=0$.
\end{example}

This result covers the case where $x_0$ is regular, we now prove Proposition \ref{prop: reached subreg} stating that the conjecture also holds for subregular elements under whose centralizers are not abelian.

\begin{proof}[Proof of Proposition \ref{prop: reached subreg}]
    Let $\mathbf{x}$ be an element of $\g_\infty$ such that $x_0$ is subregular.
    \paragraph{$1\Rightarrow2$}: This is trivial.
    \paragraph{$2\Rightarrow3$}: Let $(m,n)$ be the smallest pair of natural numbers with respect to the lexicographic order such that $m<n$  and $\Mod_{m,n}(\mathbf{x})=(n-m)\chi(\g)$. By Proposition \ref{prop:extension}, we have that $n=m+1$. Without loss of generality, we may assume that $\dim\g_{m+1}^{\mathbf{x}(m+1)}=\dim\g_m^{\mathbf{x}(m)}+\chi(\g)$. By Proposition \ref{prop:adjoint concavity}, the sequence $(\dim\g_m^{\mathbf{x}(m)})_{m\in\Natural}$ is concave. Hence, the ``Slopes Lemma'' implies that for all natural numbers $k$,
    \[\dim\g_{k+1}^{\mathbf{x}(k+1)}-\dim\g_k^{\mathbf{x}(k)}\leq\dim\g_1^{\mathbf{x}(1)}-\dim\g^{x_0}.\]
    By Theorem \ref{thm:lower bound} and Lemma \ref{lem:m=1}, we get that
    \[\chi(\g)\leq\dim\g_{k+1}^{\mathbf{x}(k+1)}-\dim\g_k^{\mathbf{x}(k)}\leq\dim\g^{x_0}=\chi(\g)+2.\]
     Since dimensions of centralizers are of the same parity, $\dim\g_{k+1}^{\mathbf{x}(k+1)}-\dim\g_k^{\mathbf{x}(k)}$ is either equal to the left-hand side or the right-hand side of the above inequality. By our choice of $m$, we know that if $k<m$ then,
     \[\dim\g_{k+1}^{\mathbf{x}(k+1)}-\dim\g_k^{\mathbf{x}(k)}=\dim\g^{x_0}.\]
     It follows from a quick recursion that for all natural numbers $k$ with $k\leq m$
     \begin{equation}
        \label{eq 1}
         \dim\g_k^{\mathbf{x}(k)}=(k+1)\dim\g^{x_0}.
     \end{equation}
     
     If $m=0$, then by Lemma \ref{lem:m=1}, we have
     \[\dim(\g_0^{x_0})^{K(x_1,\cdot)|_{\g^{x_0}}}=\dim\g_1^{\mathbf{x}(1)}-\dim\g^{x_0}=\chi(\g)<\dim\g^{x_0}.\]
     Therefore the centralizer of $K(x_1,\cdot)|_{\g^{x_0}}$ is not the whole space. Hence $\g^{x_0}$ is not abelian.
     
     If $m=1$, then Lemma \ref{lem:m=1} and \eqref{eq 1} gives
     \[\dim(\g_0^{x_0})^{K(x_1,\cdot)|_{\g^{x_0}}}=\dim\g_1^{\mathbf{x}(1)}-\dim\g^{x_0}=\dim\g^{x_0}.\]
     
     Similarily, if $m\geq2$, then by using Lemma \ref{lem: adjoint RT tech} and \eqref{eq 1} we get that
     \begin{align*}
         \dim(\g_{m-1}^{\mathbf{x}(m-1)})^{\mathbf{h}(m-1)}&=\dim\g_{m}^{\mathbf{x}(m)}-\dim\g_{m-1}^{\mathbf{x}(m-1)}+\dim\g_{m-2}^{\mathbf{x}(m-2)}\\
         &=m\dim\g^{x_0}\\
         &=\dim\g_{m-1}^{\mathbf{x}(m-1)},
     \end{align*}
     where $\mathbf{h}(m-1)$ is the restriction of $K((x_1,2x_2,...,mx_{m}),\cdot)$ to ${\g_{m-1}^{\mathbf{x}(m-1)}}.$
     
     Hence, when $m\geq1$, the obvious inclusion $(\g_{m-1}^{\mathbf{x}(m-1)})^{\mathbf{h}(m-1)}\subset\g_{m-1}^{\mathbf{x}(m-1)}$ gives the equality
     \begin{equation}
     \label{eq 2}
         (\g_{m-1}^{\mathbf{x}(m-1)})^{\mathbf{h}(m-1)}=\g_{m-1}^{\mathbf{x}(m-1)}.
     \end{equation}

     Now, by using Lemma \ref{lem: adjoint RT tech} and \eqref{eq 1}, we get that
     \begin{align*}
         \dim(\g_{m}^{\mathbf{x}(m)})^{\mathbf{h}(m)}&=\dim\g_{m+1}^{\mathbf{x}(m+1)}-\dim\g_{m}^{\mathbf{x}(m)}+\dim\g_{m-1}^{\mathbf{x}(m-1)}\\
         &=m\dim\g^{x_0}+\chi(\g)\\
         &<\dim\g_{m}^{\mathbf{x}(m)},
     \end{align*}
     where $\mathbf{h}(m)$ is the restriction of $K((x_1,2x_2,...,(m+1)x_{m+1}),\cdot)$ to ${\g_{m}^{\mathbf{x}(m)}}.$
     Hence the inclusion $(\g_{m}^{\mathbf{x}(m)})^{\mathbf{h}(m)}\subset\g_{m}^{\mathbf{x}(m)}$ is strict. This implies that there exist $\mathbf{z}(m)$ and $\mathbf{z'}(m)$ in $\g_m^{\mathbf{x}(m)}$ such that
     \[\mathbf{h}(m)([\mathbf{z}(m),\mathbf{z'}(m)])\neq0.\]
     Yet,
     
     \begin{align*}
         \mathbf{h}(m)([\mathbf{z}(m),\mathbf{z'}(m)])&=K((x_1,2x_2,...,mx_m,0),[\mathbf{z}(m),\mathbf{z'}(m)])\\&\quad\quad+K((0,...,0,(m+1)x_{m+1}),[\mathbf{z}(m),\mathbf{z'}(m)])\\
         &=\mathbf{h}(m-1)([\mathbf{z}(m-1),\mathbf{z'}(m-1)])+(m+1)K(x_{m+1},[z_0,z'_0])\\
         &=(m+1)K(x_{m+1},[z_0,z'_0]),
     \end{align*}
     because $\mathbf{h}(m-1)([\mathbf{z}(m-1),\mathbf{z'}(m-1)])=0$ by \eqref{eq 2} and the fact that $\mathbf{z}(m-1)$ and $\mathbf{z'}(m-1)$ are in $\g_{m-1}^{\mathbf{x}(m-1)}$.
     
     From this, we conclude that $(m+1)K(x_{m+1},[z_0,z'_0])\neq0$, which implies that $[z_0,z'_0]\neq0$. As $[\mathbf{z}(m),\mathbf{x}(m)]=0$, we have that $[\mathbf{z}(m),\mathbf{x}(m)]_0=[z_0,x_0]=0$ hence $z_0\in\g^{x_0}$. Similarly $z'_0\in\g^{x_0}$, hence we deduce that $\g^{x_0}$ is not abelian.
     \paragraph{$3\Rightarrow1$}: Assume that $\g^{x_0}$ is not abelian and fix $m$ a natural number. By Proposition \ref{prop:full ext}, if there exists a natural number $l<m$ such that $\dim\g_{l+1}^{\mathbf{x}(l+1)}-\dim\g_{l}^{\mathbf{x}(l)}=\chi(\g)$, then for all natural numbers $n\geq m$, $\Mod_{m,n}(\mathbf{x})=(m-n)\chi(\g)$, and we are done. So we can assume that for all natural numbers $l<m$, $\dim\g_{l+1}^{\mathbf{x}(l+1)}-\dim\g_{l}^{\mathbf{x}(l)}>\chi(\g)$. As we have seen above, this implies that $\dim\g_{l+1}^{\mathbf{x}(l+1)}-\dim\g_{l}^{\mathbf{x}(l)}=\chi(\g)+2=\dim\g^{x_0}$. It follows that for all natural numbers $l$ with $l\leq m$
     \[\dim\g_{l}^{\mathbf{x}(l)}=(l+1)\dim\g^{x_0}.\]

    Let $\mathbf{y}$ be an element of $\g_\infty$ such that $\mathbf{y}(m)=\mathbf{x}(m)$ and $y_{m+1}$ is not in $[\g^{x_0},\g^{x_0}]^\perp$, where orthogonality is with respect to $K$. Such a $\mathbf{y}$ exists as $\g^{x_0}$ is not abelian and $K$ is non-degenerate.
    By Lemma \ref{lem: adjoint RT tech}, we have

    $$\dim\g_{m+1}^{\mathbf{y}(m+1)}=\dim(\g_m^{\mathbf{x}(m)})^{\mathbf{h}(m)}+\chi(\g)+2,$$
    where $\mathbf{h}(m)$ is the restriction of $K((x_1,2x_2,...,mx_{m},(m+1)y_{m+1}),\cdot)$ to ${\g_m^{\mathbf{x}(m)}}$.
    
    Let $z$ and $z'$ be elements of $\g^{x_0}$ such that $K(y_{m+1},[z,z'])\neq0$. Consider the injection $\iota_{m}:\g_{m-1}^{\mathbf{x}(m-1)}\rightarrow\g_{m}^{\mathbf{x}(m)}$ that maps $(z_0,...,z_{m-1})$ to $(0,z_0,...,z_{m-1})$ and the surjection $\pi_m:\g_{m}^{\mathbf{x}(m)}\rightarrow\g^{x_0}$ which maps $(z_0,...,z_m)$ to $z_0$. We have that $\iota_m$ is a bijection onto $\ker\pi_{m}$, so $\dim\operatorname{Im}(\pi_{m})=\dim\g_{m}^{\mathbf{x}(m)}-\dim\g_{m-1}^{\mathbf{x}(m)}=\dim\g^{x_0}$. Hence $\pi_{m}$ is surjective onto $\g^{x_0}$. Therefore, there exists $\mathbf{z'}(m)\in\g_m^{\mathbf{x}(m)}$ such that $z'_0=z'$ . Consequently,
    
    $$\mathbf{h}([(0,...,0,z),\mathbf{z'}(m)])=(m+1)K(y_{m+1},[z,z'])\neq0.$$

    As $(0,...,0,z)$ is in $\g_m^{\mathbf{x}(m)}$, we deduce that $(\g_m^{\mathbf{x}(m)})^{\mathbf{h}(m)}\subsetneq\g_m^{\mathbf{x}(m)}$. Then
    $$\dim\g_{m+1}^{\mathbf{y}(m+1)}<\dim\g_m^{\mathbf{x}(m)}+\chi(\g)+2.$$

  As the dimensions of centralizers have the same parity, we get that
    $$\dim\g_{m+1}^{\mathbf{y}(m+1)}\leq\dim\g_m^{\mathbf{x}(m)}+\chi(\g).$$

     By Theorem \ref{thm:lower bound}, we have $\chi(\g)\leq\Mod_{m,m+1}(\mathbf{x})$. Hence we deduce the equality $\Mod_{m,m+1}(\mathbf{x})=\chi(\g)$.
    
    Finaly, we conclude by Proposition \ref{prop:extension} that $\Mod_{m,n}(\mathbf{x})=(m-n)\chi(\g)$ for all natural numbers $n\geq m$.
\end{proof}

\begin{example}
    Let $\mathfrak{n}$ be the Lie algebra given in Example \ref{example: list}. By a direct computation we have that $\chi(\mathfrak{n})=4$, the set $\mathfrak{n}_\mathrm{reg}$ of regular elements of $\mathfrak{n}$ is $\Complex a_1\oplus\Complex a_2\oplus\Complex a_3\setminus\{0\}$ and the set of its subregular elements is $\mathfrak{n}\setminus\mathfrak{n}_\mathrm{reg}$. Subregular elements have centralizers $\mathfrak{n}$ which are not abelian. By Theorem \ref{thm:reached reg} and Proposition \ref{prop: reached subreg} we conclude that for all elements $\mathbf{x}$ of $\mathfrak{n}_\infty$ and all pairs of natural numbers $(m,n)$ with $m\leq n$, $\Mod_{m,n}(\mathbf{x})=(n-m)\chi(\mathfrak{n})=4(n-m).$
    As such, even though it is not reductive, $\mathfrak{n}$ satifies the conclusion of Conjecture \ref{conj: main}.

    In Example \ref{example: s}, the element $0$ is subregular, we recover the equality given when $x_0=0$.
\end{example}

\begin{rem}
    Note that $\chi(\g)+2$ is the minimal dimension for the centralizer of a non regular element and that having an abelian centralizer is a closed condition. Hence, the set of subregular elements of $\g$ whose centralizers are not abelian is an open subset of the set of non-regular elements in $\g$.
\end{rem}

\begin{cor}
    \label{cor: not reached subreg}
    Let $\mathbf{x}$ be an element of $\g_\infty$ such that $x_0$ is subregular, then the following are equivalent:
    \begin{enumerate}
        \item For all pairs of natural numbers $(m,n)$ with $m\leq n$, we have that
        
        $\Mod_{m,n}(\mathbf{x})=(n-m)(\chi(\g)+2).$
        \item There exists a pair of natural numbers $(m,n)$ with $m<n$ such that $\Mod_{m,n}(\mathbf{x})=(n-m)(\chi(\g)+2)$.
        \item The centralizer $\g^{x_0}$ is an abelian Lie algebra.
    \end{enumerate}
\end{cor}
\begin{proof}
    Together with Theorem \ref{thm:lower bound}, the contraposition of Propostion \ref{prop: reached subreg} yields that the following are equivalent:
    \begin{enumerate}
        \item For all pairs of natural numbers $(m,n)$ with $m<n$ we have that
        
        $\Mod_{m,n}(\mathbf{x})>(n-m)\chi(\g).$
        \item There exists a pair of natural numbers $(m,n)$ with $m<n$ such that $\Mod_{m,n}(\mathbf{x})>(n-m)\chi(\g)$.
        \item The centralizer $\g^{x_0}$ is an abelian Lie algebra.
    \end{enumerate}
    Let $m\leq n$ be two natural numbers and $\mathbf{y}\in\g_\infty$ such that $\mathbf{y}(m)=\mathbf{x}(m)$. By Theorem \ref{prop:adjoint concavity}, the sequence $(\dim\g_k^{\mathbf{y}(k)})_{k\in\Natural}$ is concave. As such, $\dim\g_n^{\mathbf{y}(n)}-\dim\g_m^{\mathbf{y}(m)}\leq(n-m)(\dim\g_1^{\mathbf{y}(1)}-\dim\g^{y_0})$. By Lemma \ref{lem:m=1} $\dim\g_1^{\mathbf{y}(1)}-\dim\g^{y_0}=\dim(\g^{y_0})^{K(y_1,\cdot)|_{\g^{y_0}}}\leq\dim\g^{y_0}$. As $y_0=x_0$ and $x_0$ is subregular, we get
    $\dim\g_n^{\mathbf{y}(n)}-\dim\g_m^{\mathbf{y}(m)}\leq(n-m)\dim\g^{x_0}=(n-m)(\chi(\g)+2)$. By definition of $(m,n)$-modality, this implies that $\Mod_{m,n}(\mathbf{x})\leq(n-m)(\chi(\g)+2)$ and we are done.    
\end{proof}

\subsection{The reductive case}
In this subsection, the Lie algebra $\g$ is assumed to be reductive.
\subsubsection{The $(0,n)$-modality and Elashvili's conjecture}
\label{subsub:Elash}
We now give the proof of Theorem \ref{thm:main} and the equivalence of Theorem \ref{thm: 01} with the usual formulation of Elashvili's conjecture (Theorem \ref{conj:Elashvili}).

\begin{proof}[Proof of Theorem \ref{thm:main}]
    By Lemma \ref{lem:m=1}, we have that $\dim\g_1^{(x_0,y_1)}=\dim\g^{x_0}+\dim(\g^{x_0})^{K(y_1,\cdot)|_{\g^{x_0}}}$ for all $y_1$ in $\g$. The map $y\mapsto K(y,\cdot)$
    from $\g$ to $\g^*$ is surjective so $y\mapsto K(y,\cdot)|_{\g^{x_0}}$
    from $\g$ to $(\g^{x_0})^*$ also is. As such $\Mod_{0,1}(\mathbf{x})=\min_{y_1\in\g}\dim(\g^{x_0})^{K(y_1,\cdot)|_{\g^{x_0}}}=\chi(\g^{x_0})$ which prove the equivalence between Theorem \ref{thm: 01} and Theorem \ref{conj:Elashvili}.

    We conclude by Proposition \ref{prop:extension} that $\Mod_{0,n}(\mathbf{x})=n\chi(\g)$ for all natural numbers $n$.
\end{proof}

\subsubsection{Certain classes of elements}
The proof of Theorem \ref{thm:red reg/sub} follows immediately from the reductiveness of $\g$ and our previous results.

\begin{proof}[Proof of Theorem \ref{thm:red reg/sub}]
    Centralizers of subregular elements in a reductive Lie algebra are not abelian (see \cite{TauvelYu2005} 35.6.8), as such Proposition \ref{prop: reached subreg} and Theorem \ref{thm:reached reg} imply Theorem \ref{thm:red reg/sub}.
\end{proof}

\begin{example}
    Every non-zero element of $\mathfrak{sl}_2(\Complex)$ is regular and $0$ is subregular, we deduce that Conjecture \ref{conj: main} is satisfied for $\mathfrak{sl}_2(\Complex)$.

    Every element of $\mathfrak{sl}_3(\Complex)$ is either $0$, subregular or regular. Let $\mathbf{x}$ be an element of $\mathfrak{sl}_3(\Complex)_\infty$. By Theorem \ref{thm:red reg/sub}, if $x_0\neq0$ then, for all pairs of natural numbers $(m,n)$ with $m\leq n$ we have $\Mod_{m,n}(\mathbf{x})=(n-m)\chi(\mathfrak{sl}_3(\Complex))=2(n-m).$ When $x_0=0$, we can apply Proposition \ref{prop:shift} recursively to get that the equality still holds. We deduce that Conjecture \ref{conj: main} is satisfied for $\g=\mathfrak{sl}_3(\Complex).$
\end{example}

In the study of centralizers in a reductive Lie algebra, many problems can be reduced to the study of nilpotent elements in a simple Lie algebra. This is due to the fact that centralizers of semisimple elements in a Lie algebra are reductive and that given a Jordan decomposition $x=x_\mathrm{ss}+x_{\mathrm{nil}}$ where $x_\mathrm{ss}$ is the semisimple part and $x_{\mathrm{nil}}$ the nilpotent part, we have that $\g^x=(\g^{x_\mathrm{ss}})^{x_{\mathrm{nil}}}$.

We shall show how a similar reduction can be realized in the generalized Takiff Lie algebras of a reductive Lie algebras. For that, we need to prove two technical lemmas. The first one being a more precise description of the relationship between the Jordan decomposition and the adjoint action of an element, the second one on the action of the algebraic adjoint group.

\begin{lem}
\label{lem:Jordan decomp}
    Let $x$ be an element of $\g$ and $x=x_\mathrm{ss}+x_{\mathrm{nil}}$ its Jordan decomposition. Then
    $$\g^x=(\g^{x_\mathrm{ss}})^{x_{\mathrm{nil}}}\text{ and }[\g,x]=[\g^{x_\mathrm{ss}},x_{\mathrm{nil}}]\oplus[\g,x_\mathrm{ss}].$$
\end{lem}

\begin{proof}
    By the properties of the Jordan decomposition in a reductive Lie algebra, the elements $x$, $x_\mathrm{ss}$ and $x_{\mathrm{nil}}$ commute, and as $x_\mathrm{ss}$ is semisimple, we have that
    $$\g=\g^{x_\mathrm{ss}}\oplus[\g,x_\mathrm{ss}].$$
    Hence we deduce that
    \begin{align*}
        [\g,x]&=[\g^{x_\mathrm{ss}},x_{\mathrm{nil}}]+[[\g,x_\mathrm{ss}],x]\\
        &=[\g^{x_\mathrm{ss}},x_{\mathrm{nil}}]+[[\g,x],x_\mathrm{ss}].
    \end{align*}
    As $\g^{x_\mathrm{ss}}$ is a Lie subalgebra of $\g$ and $x_{\mathrm{nil}}$ is in $\g^{x_\mathrm{ss}}$, we have $[x_{\mathrm{nil}},\g^{x_\mathrm{ss}}]\subset{\g^{x_\mathrm{ss}}}$, hence the sum above is direct.

    It follows that
    \begin{align*}
        \dim[[\g,x],x_\mathrm{ss}]&=\dim[\g,x]-\dim[\g^{x_\mathrm{ss}},x_{\mathrm{nil}}]\\
        &=\dim\g-\dim\g^x-(\dim\g^{x_\mathrm{ss}}-\dim(\g^{x_\mathrm{ss}})^{x_{\mathrm{nil}}})\\
        &=\dim\g-\dim\g^{x_\mathrm{ss}}\\
        &=\dim[\g,x_\mathrm{ss}].
    \end{align*}

    Consequently, $[[\g,x],x_\mathrm{ss}]=[\g,x_\mathrm{ss}]$ and the result follows.
\end{proof}

\begin{lem}
\label{lem:proj}
    Let $\mathbf{x}$ be an element of $\g_\infty$, $V$ a subspace of $[\g,x_0]$ and $S$ a subspace of $\g$ such that $\g=V+S$. For all natural numbers $m$, there exists an element $\Tilde{\mathbf{x}}$ of $\g_\infty$ verifying the following conditions:
    \begin{center}
        \begin{itemize}
            \item $\Tilde{\mathbf{x}}(m)$ is in the $\G_m$-orbit of $\mathbf{x}(m)$, where $\G_m$ denotes the algebraic adjoint group of $\g_m$.
            \item $\Tilde{x}_0=x_0$, and for all $k\geq 1$, we have that $\Tilde{x}_k$ is in $S$.
        \end{itemize}
    \end{center} 
\end{lem}

\begin{proof}
    We shall show by induction on $l$ where $0 \leq l\leq m$ that there exists $\Tilde{\mathbf{y}}\in\g_\infty$ such that $\mathbf{x}(m)\underset{\G_m}{\sim}\Tilde{\mathbf{y}}(m)$, $\Tilde{y}_0=x_0$ and  for all $1\leq k\leq l$, $\Tilde{y}_k$ is in $S$. 

    The case $l=0$ is immediate by taking $\Tilde{\mathbf{y}}=\mathbf{x}$.

    Let $0 \leq l\leq m-1$ and assume that the induction hypothesis is satisfied for $l$. Let $\Tilde{\mathbf{y}}\in\g_\infty$ be such that $\mathbf{x}(m)\underset{\G_m}{\sim}\Tilde{\mathbf{y}}(m)$, $\Tilde{y}_0=x_0$ and for all $1\leq k\leq l$, $\Tilde{y}_k$ is in $S$.
    
    Let $\Tilde{y}_{l+1}^S\in S$ and $\Tilde{y}_{l+1}^V\in V$ be such that $\Tilde{y}_{l+1}=\Tilde{y}_{l+1}^S+\Tilde{y}_{l+1}^V$. As $V$ is a subspace of $[\g,x_0]$, there exists an element $y'$ of $\g$ such that $\Tilde{y}_{l+1}^V=[y',x_0]$. The group $\G_m$ contains elementary automorphisms of $\g_m$ (see \cite{TauvelYu2005} 24.8.2), in particular it contains:
    $$g:=\exp{\operatorname{ad}_{\g_m}(-y'\otimes T^{l+1})}.$$

    Therefore
    $$\Tilde{\mathbf{z}}(m):=g\cdot\Tilde{\mathbf{y}}(m)\underset{\G_m}{\sim}\mathbf{x}(m).$$

    A direct computation shows that for all $0\leq k\leq l$, we have $ \Tilde{z}_k=\Tilde{y}_k$ and $\Tilde{z}_{l+1}=\Tilde{y}_{l+1}-[y',x_0]=\Tilde{y}_{l+1}^s$. Hence the result follows by induction up to $m$ and then choosing $\Tilde{x}_k=0$ for $k>m$.
\end{proof}

From the last two lemmas, we deduce the following reduction result.
\begin{prop}
\label{prop: nilp red}
    Let $\mathbf{x}$ be an element of $\g_\infty$, $n$ a natural number and $\G_n$ the algebraic adjoint group of $\g_n$, $x_0=x_\mathrm{ss}+x_{\mathrm{nil}}$ the Jordan decomposition of $x_0$ and $\mathfrak{s}$ the centralizer of $x_\mathrm{ss}$ in $\g$. There exists an element $\Tilde{\mathbf{x}}$ of $\mathfrak{s}_\infty\subset\g_\infty$ such that $\Tilde{x}_0=x_{\mathrm{nil}}$ and $\mathfrak{g}_n^{\Tilde{\mathbf{x}}(n)}$ is $\G_n$-conjugate to $\g_n^{\mathbf{x}(n)}$. Moreover, $\mathfrak{g}_n^{\Tilde{\mathbf{x}}(n)}=\mathfrak{s}_n^{\Tilde{\mathbf{x}}(n)}$.

    In particular, for all $0\leq m\leq n$, the $(m,n)$-modality of $\mathbf{x}$ as an element of $\g_\infty$ is equal to the $(m,n)$-modality of $\Tilde{\mathbf{x}}$ as an element of $\mathfrak{s}_\infty$.
\end{prop}

\begin{proof}
    By Lemma \ref{lem:Jordan decomp}, we can apply Lemma \ref{lem:proj} with $V=[\g,x_\mathrm{ss}]$ and $S=\mathfrak{s}$. So there exists $\Tilde{\mathbf{x}}\in \mathfrak{s}_\infty$ such that $\Tilde{x}_0=x_0$ and $g\cdot\Tilde{\mathbf{x}}(m)=\mathbf{x}(m)$ for some $g\in\G_m$. 

    This implies $g\cdot\g_m^{\Tilde{\mathbf{x}}(m)}=\g_m^{\mathbf{x}(m)}$. We claim that $\g_m^{\Tilde{\mathbf{x}}(m)}\subset\mathfrak{s}_m$.
    
    Indeed, let $\mathbf{y}(m)$ be an element of $\g_m^{\Tilde{x}(m)}$. In particular, $[y_0,\Tilde{x}_0]=0.$ As $x_0=\Tilde{x}_0$, we have $y_0\in\g^{x_0}$ and Lemma \ref{lem:Jordan decomp} implies that $y_0=\mathfrak{s}^{x_{\mathrm{nil}}}.$

    Suppose that we have shown that $\mathbf{y}(l)$ is in $\mathfrak{s}_l$ for some integer $l$ between $0$ and $m-1$. From $[\mathbf{y}(m),\mathbf{\Tilde{x}}(m)]=0$, we obtain that
    $$\sum_{i=0}^{l+1}{[y_i,\Tilde{x}_{l+1-i}]}=0.$$

    In particular, as $\mathbf{\Tilde{x}}\in\mathfrak{s}_\infty$, we have $[y_{l+1},x_0]\in\mathfrak{s}.$

    Now, as $[y_{l+1},x_0]=[y_{l+1},x_\mathrm{ss}]+[y_{l+1},x_{\mathrm{nil}}]$ and $\mathfrak{s}\cap[\g,x_\mathrm{ss}]=\{0\}$, we deduce that $[y_{l+1},x_\mathrm{ss}]=0$, hence $y_{l+1}\in\mathfrak{s}$.

    By a finite iteration we obtain our claim. Now, as $x_\mathrm{ss}$ is central in $\mathfrak{s}$, we have $\mathfrak{s}_m^{\Tilde{\mathbf{x}}(m)}=\mathfrak{s}_m^{(x_{\mathrm{nil}},\Tilde{x}_1,..,\Tilde{x}_m)}$, hence $(x_{\mathrm{nil}},\Tilde{x}_1,\Tilde{x}_2,...)$ satisfies the required conditions.
\end{proof}

Looking at the first step of the recursion in the proof of Lemma \ref{lem:proj} with $V=[\g,x_\mathrm{ss}]$ and $S=\mathfrak{s}$, we can explicit  $\Tilde{x}_1$ to deduce the following reduction for $(1,n)$-modality

\begin{cor}
    Let $\mathbf{x}$ be an element of $\g_\infty$, $x_0=x_\mathrm{ss}+x_\mathrm{nil}$ the Jordan decomposition of $x_0$, $\mathfrak{s}$ the centralizer of $x_\mathrm{ss}$ in $\g$. Let $\Tilde{\mathbf{x}}\in\mathfrak{s}_\infty$ such that $\Tilde{x}_0=x_\mathrm{nil}$ and $\Tilde{x}_1$ is the projection of $x_1$ on $\mathfrak{s}$ parallel to $[\g,x_\mathrm{ss}]$ the $(1,n)$-modality of $\mathbf{x}$ as an element of $\g_\infty$ is equal to the $(1,n)$-modality of $\Tilde{\mathbf{x}}$ as an element of $\mathfrak{s}_\infty$.
\end{cor}

Using this corollary and Proposition \ref{prop:shift}, Theorem \ref{thm:main} implies

\begin{cor}
     Let $\mathbf{x}$ be an element of $\g_\infty$ such that $x_0$ is semisimple then for all natural numbers $n\geq1$, we have
     \[\Mod_{1,n}(\mathbf{x})=(n-1)\chi(\g).\]
\end{cor}

\begin{ack*}
    The author wishes to thank Rupert Yu for fruitful discussions on this paper and the anonymous reviewer whose suggestions greatly helped improve its presentation.
\end{ack*}

\bibliographystyle{plainnat}
\bibliography{references.bib}

\end{document}